\documentclass[10pt]{article}

\usepackage{amssymb}
\usepackage{amscd,enumerate,amsfonts,calc,amsmath,verbatim,xypic}
\usepackage[all]{xy}

\usepackage{amsthm}

\providecommand{\vone}{\vskip 1ex}
\providecommand{\vtwo}{\vskip 2ex}
\providecommand{\hhalf}{\mbox{\hspace{0.5em}}}

\theoremstyle{plain}
\newtheorem{theorem}{Theorem}[section]
\newtheorem{corollary}[theorem]{Corollary}
\newtheorem{proposition}[theorem]{Proposition}
\newtheorem{question}[theorem]{Question}
\newtheorem{lemma}[theorem]{Lemma}
\newtheorem{remark}[theorem]{{Remark}}
\newtheorem{setup}[theorem]{Setup}
\newtheorem{definition}[theorem]{Definition}
\newtheorem{example}[theorem]{Example}

\providecommand{\rar}{\rightarrow}
\providecommand{\lrar}{\longrightarrow}
\providecommand{\into}{\hookrightarrow}
\providecommand{\onto}{\lrar\!\!\!\!\rar}
\providecommand{\inc}{\subseteq}

\newcommand{\ses}[5]{\ensuremath
{0 \rar {#1} \overset{{#4}}\rar {#2} \overset{{#5}}\rar {#3} \rar 0 }}

\renewcommand{\vec}[3]{\ensuremath{#1_#2, \ldots, #1_#3 }}

\providecommand\adj{\text{\it adj}}

\renewcommand\dim{\text{\rm dim}}
\providecommand\Hom{\text{\rm Hom}}

\providecommand\Ker{\text{\rm ker}}
\providecommand\ker{\text{\rm Ker}}
\providecommand\min{\text{\rm min}}
\providecommand\soc{\text{\rm soc}}

\providecommand{\Q}{{\mathbb Q}}
\renewcommand\-{{\_\!\_}}
\providecommand{\sk}{{\ensuremath{\sf k }}}
\providecommand{\ma}{{\mathfrak a }}
\providecommand{\mb}{{\mathfrak b }}
\providecommand{\mc}{{\mathfrak c }}
\providecommand{\md}{{\mathfrak d }}

\providecommand{\m}{{\mathfrak m }}

\begin{document}

\title{\large The Gorenstein Colength of an Artinian Local Ring}
\author{H. Ananthnarayan}
\date{\today}
\maketitle

\abstract{
\noindent
In this paper, we make the notion of approximating an Artinian local ring by a Gorenstein Artin local ring precise using the concept of Gorenstein colength. We also answer the question of when the Gorenstein colength is at most two.
}

\section{\large Introduction}

Let $T$ be a commutative Noetherian ring and $\mb$ an ideal in $T$ such that $R := T/\mb$ is Cohen-Macaulay. A problem of interest to many mathematicians is finding Gorenstein rings $S$ mapping onto the Cohen-Macaulay ring $R$. We are interested not only in finding such a Gorenstein ring, but also find one as ``close'' to $R$ as possible. More specifically, the question we would like to answer is the following:

Given an Artinian local ring $(R,\m,\sk)$, how ``close'' can one get to $R$ by a Gorenstein Artin local ring? In order to make this notion precise, we introduce a number called the Gorenstein colength of $R$ in Definition \ref{D0}. We use the following notation throughout the paper.

\begin{setup} \label{S1} \hfill{}

\noindent
{\rm
1. Let $(R,\m,\sk)$ be an Artinian local ring and $\omega_R$ (or simply $\omega$) be the canonical module of $R$. Since $R$ is Artinian, $\omega$ is the same as the injective hull over $R$, of the residue field $\sk$. By $(\-)^*$ and $(\-)^\vee$, we mean $\Hom_R(\-,R)$ and $\Hom_R(\-,\omega)$ respectively.

\vskip 2pt
\noindent
2. By Cohen's Structure Theorem, we can write $R \simeq T/\mb$, where $(T,\m_T,\sk)$ is a regular local ring and $\mb$ is an $\m_T$-primary ideal. By $\ \bar{}\ $, we mean going modulo $\mb$.
}\end{setup}

\begin{definition}\label{D0}{\rm 
Let $(R,\m,\sk)$ be an Artinian local ring. Define the Gorenstein colength of $R$, denoted $g(R)$ as:\\
$g(R) = \min\{\lambda(S) - \lambda(R) : S$ is a Gorenstein Artin local ring 
mapping onto $R\},$\\ where $\lambda(\-)$ denotes length. 
}\end{definition}

The number $g(R)$ gives a numerical value to how close one can get to an Artinian local ring $R$ by a Gorenstein Artin local ring. We do not require the embedding dimension of $S$ to be the same as that of $R$.

It is clear that $g(R)$ is zero if and only if $R$ is Gorenstein. Observe that $g(R) = 1$ if and only if $R$ is not Gorenstein and $R \simeq S/\soc(S)$ for a Gorenstein Artin ring $S$. W. Teter gives a characterization for such rings in his paper \cite{T}. In their paper \cite{HV}, C. Huneke and A. Vraciu refer to these rings as Teter's rings.

\noindent
With notation as in Setup \ref{S1}.1, Teter's theorem states:

\begin{theorem}[Teter]\label{Teter}
Let $(R,\m,\sk)$ be an Artinian ring. Then the following are equivalent:\\
{\rm i)} $g(R) \leq 1$.\\
{\rm ii)} Either $R$ is Gorenstein or there is an isomorphism $\m \overset{\phi}\lrar \m^\vee$ such that $\phi(x)(y) = \phi(y)(x)$, for every $x,y$ in $\m$. 
\end{theorem}

{\rm The commutativity condition on the map $\phi$ in (ii) of Theorem \ref{Teter} is an awkward technical condition. The following theorem (\cite{HV}, Theorem 2.5), of Huneke and Vraciu is an improvement of Theorem \ref{Teter}, which gets rid of Teter's technical condition on the map $\phi$. However, they need to assume that 2 is invertible in $R$ and $\soc(R) \inc \m^2$. \vtwo}

\begin{theorem}[Huneke-Vraciu]\label{AC}
Let $(R,\m,\sk)$ be an Artinian ring such that $1/2 \in R$, $\soc(R) \inc \m^2$. 
Then the following are equivalent:\\
{\rm i)} $g(R) \leq 1$.\\
{\rm ii)} Either $R$ is Gorenstein or $\m \simeq \m^\vee$.
\end{theorem}

A natural question one can ask is whether we can characterize Artinian local rings whose Gorenstein colength is at most two. In section 5, we prove the main theorem in this paper (Theorem \ref{T1}), which is an extension of Teter's theorem. We also extend the Huneke-Vraciu theorem and as a consequence, show the following:

\vtwo

\noindent
{\bf Theorem \ref{T2} }{\it With notation as in Setup {\rm \ref{S1}}, suppose that $\mb \inc \m_T^6$. Moreover, assume that $2$ is invertible in $R$. Then the following are equivalent:\\
{\rm i)} $g(R) \leq 2$.\\
{\rm ii)} There exists an ideal $\overline \ma \inc R$ with $\lambda(R/\overline \ma) \leq 2$ such that $\overline{\ma} \simeq \overline{\ma}^\vee$.
 }
\vtwo
We record some properties of Gorenstein colength in section 2. In section 3, we investigate the role played by self-dual ideals in the study of Gorenstein colength. As can be seen in Lemma \ref{L1}, maps from the canonical module $\omega$ to $R$ are closely related to self-dual ideals. We study these maps via an involution on $\omega^*$ in section 4 and as an application, prove Theorem \ref{T:1}. This theorem gives an upper bound on the Gorenstein colength of rings which have an algebra retract with respect to a self-dual ideal.

\section{\large  More on Gorenstein Colength}

\begin{proposition}\label{P1} Let $(R,\m,\sk)$ be an Artinian local ring. Then $g(R)$ is finite. 
\end{proposition}

\begin{proof}
As in Setup \ref{S1}.2, write $R \simeq T/\mb$ where $(T,\m_T,\sk)$ is a regular local ring  and $\mb$ is an $\m_T$-primary ideal. If $\dim(T) = d$, choose a regular sequence \vec{x}{1}{d} in $\mb$ and set $S:= T/(\vec{x}{1}{d})$. Then $S$ is a Gorenstein Artin local ring mapping onto $R$. 
Thus $g(R) \leq \lambda(S) - \lambda(R)$ which is finite. \end{proof}

 In fact, the above proof shows that if $\sk$ is infinite, then by choosing $(\vec{x}{1}{d})$ to be a minimal reduction of $\mb$, we see that $g(R) \leq e(\mb) - \lambda(R),$ where $e(\mb) = \lambda(S)$ is the multiplicity of $\mb$.\\
 
\begin{proposition}\label{P2} Let $(R,\m,\sk)$ be an Artinian local ring. Then $g(R) \leq \lambda(R).$ \end{proposition}
\begin{proof}

Let $\omega$ be the canonical module of $R$. We can define a ring structure on $S := R \oplus \omega$ using Nagata's principle of idealization. It is a well-known fact (eg. \cite{BH}, Theorem 3.3.6) that $S$ is Gorenstein. Since $\lambda(S) = 2 \lambda (R)$, and $S$ maps onto $R$ via the natural projection, 
$g(R) \leq \lambda(R)$. 
\end{proof}

\begin{example}\label{E1}{\rm
In this example, we see that $g(R) < \min \{e(\mb) - \lambda(R), \lambda(R)\}$ with notation as above. 

Let $T = \Q[X,Y,Z]$, $\mb = (X,Y,Z)^2$ and $R = T/\mb$. We have $e(\mb) = 8$ and $\lambda(R) = 4$. Let $\mc = (X^2 - Y^2,X^2 - Z^2, XY, XZ, YZ)$ and $S = T/\mc$. Then $S$ is a Gorenstein Artin ring that maps onto $R$. Since $\lambda(S) = 5$ and $R$ is not Gorenstein, we see that $g(R) = 1$. }
\end{example}
\vone

Let $(R,\m,\sk)$ be an Artinian local ring. The main questions one would like to answer are the following: 

\noindent
a) How does one intrinsically compute $g(R)$? 

\noindent
b) How does one construct a Gorenstein Artin local ring $S$ mapping onto $R$ such that $\lambda(S) - \lambda(R) = g(R)$?

\section{\large Gorenstein Colength and Self-dual Ideals}

\begin{definition}{\rm
Let $(R,\m,\sk)$ be an Artinian local ring with canonical module $\omega$. Recall that by $(\-)^\vee$, we mean $\Hom_R(\-,\omega)$. We say that an ideal $\ma \inc R$ is self-dual if $\ma \simeq \ma^\vee$.
}\end{definition}

As one can see from the Huneke-Vraciu theorem and Theorem \ref{T2}, Gorenstein colength is closely related to self-dual ideals. 

\begin{definition}\label{D2}{\rm 
We say that the map $f: \omega \lrar \ma$ $($resp. $\phi: \ma \lrar \ma^\vee)$ satisfies {\it Teter's condition} if the commutativity condition $f(x)y = f(y)x$ for all $x$,$y \in \omega$ $($resp. $\phi(x)(y) = \phi(y)(x)$ for all $x$, $y \in \ma)$ is satisfied.
}\end{definition}

\begin{remark}\label{R0.75}{\rm
Let $\ma \overset{i}\into R$. This induces a surjective map $\omega \overset{i^\vee}\onto \ma^\vee$, 
such that for every $a \in \ma$ and $u \in \omega$, $i^\vee(u)(a) = au.$
}\end{remark}

\noindent
The following lemma tells us how self-dual ideals arise. 

\begin{lemma}\label{L1}
Let $\ma$ be an ideal in $R$. The following are equivalent:\\
{\rm i)} There is an isomorphism $\phi: \ma \overset{\sim}\lrar \ma^\vee$.\\
{\rm ii)} There is a surjective map $\omega \overset{f}\onto \ma$ such that $\ker(f) = (0 :_\omega \ma)$.  \\
Moreover $\phi$ satisfies Teter's condition if and only if $f$ satisfies Teter's condition.
\end{lemma}

\begin{proof} \hfill{}\\
(i) $\Rightarrow$ (ii): Apply $\Hom(\-,\omega)$ to the short exact sequence $\ses{\ma}{R}{R/\ma}{i}{}$ to get the short exact sequence $\ses{(0 :_{\omega} \ma)}{\omega}{{\ma^\vee}}{}{{i^\vee}}$. Let $f = \phi^{-1} \circ i^\vee : \omega \onto \ma$. Since $\phi$ is an isomorphism, $\ker(f) = \ker(i^\vee) = (0 :_\omega \ma)$.

\noindent
(ii) $\Rightarrow$ (i): Comparing the short exact sequences \ses{{(0:_\omega \ma)}}{\omega}{{\ma ^\vee}}{}{{i^\vee}} and \ses{{\ker(f)}}{\omega}{{\ma}}{}{f}, we get an isomorphism $\phi : \ma \overset{\sim}\lrar \ma^\vee$. 

For $u$, $v \in \omega$, it follows from Remark \ref{R0.75} that $\phi(f(u))(f(v)) = f(v)u$. Thus, $f$ satisfies Teter's condition if and only if $\phi$ does, proving the last part of the lemma.
\end{proof}

Let us now see what happens when a Gorenstein Artin local ring $S$ maps onto the given Artinian local ring $R$. We summarize our observations in the next proposition. These lead to lower bounds on $g(R)$.
  
\begin{proposition}\label{KO} Let $(S,\m_S,\sk)$ be a Gorenstein Artin local ring and $(R,\m,\sk)$ be an Artinian local ring with canonical module $\omega$. Let $\psi:S \lrar R$ be a surjective ring homomorphism such that $\ker(\psi) = \mb$. Then\\
{\rm 1)} $\omega$ is isomorphic to an ideal in $S$,\\
{\rm 2)} $\ker(f) \cdot f(\omega) = 0$ where $f = \psi|_\omega$ and\\
{\rm 3)} $f:\omega \lrar R$ satisfies Teter's condition.
\end{proposition}

\begin{proof}\hfill{}\\
1) $S$ is a Gorenstein ring of the same dimension mapping onto $R$. Therefore $\omega \simeq \Hom_S(R,S)$ $\simeq (0 :_S \mb) \inc S.$

\noindent
2) We have $\omega \simeq (0 :_S \mb$), $f(\omega) \simeq ((0:_S\mb) + \mb)/\mb$ and $\ker(f) \simeq \mb \cap\ (0:_S \mb)$. Hence $\ker(f) \cdot f(\omega) = 0$.

\noindent
3) Since the elements of $\omega$ can be identified with elements of $S$, for any $x$, $y$ in $\omega$, $f(x)y = f(y) x$.
\end{proof}

\begin{corollary}\label{C0.25}
With notation as in Proposition {\rm \ref{KO}}, the ideal $\ma := f(\omega)$ is self-dual.
\end{corollary}

\begin{proof}
In order to prove that $\ma$ is self-dual, by Lemma \ref{L1}, we only need to show that $\ker(f) = (0 :_\omega \ma)$. Since $\ker(f) \inc (0 :_\omega \ma)$ by Proposition \ref{KO}.2, it is enough to prove that their lengths are the same. Since $(0 :_\omega \ma) \simeq (R/\ma)^\vee$, $\lambda(0 :_\omega \ma) = \lambda(R/\ma) = \lambda(\omega) - \lambda(\ma) = \lambda(\ker(f))$, finishing the proof. 
\end{proof}

\begin{lemma}\label{L2} With notation as in Proposition {\rm \ref{KO}}, $\lambda(S) - \lambda(R) \geq\lambda(R/\psi(\omega))$.\\
Moreover equality holds, i.e., $\lambda(S) - \lambda(R) = \lambda(R/\psi(\omega))$ if and only if $\mb^2 = 0$.
\end{lemma}
\begin{proof} Let $\ma = \psi(\omega)$. The  isomorphism $\omega \simeq (0:_S\mb)$ in $S$ yields $\lambda(R) = \lambda((0:_S\mb)$. Since $S/(\omega + \mb) \simeq R/\ma$, the lemma is proved if we show $\lambda(S) - \lambda(0:_S \mb) \geq \lambda(R/\ma),$ i.e., if $\lambda(S/(0:_S \mb)) \geq \lambda(S/((0:_S\mb) + \mb))$.

But this is always true. Moreover, equality holds, i.e., $\lambda(S) - \lambda(R) = \lambda(\omega) - \lambda(\ma)$ if and only if $\mb \inc (0:_S\mb)$, i.e., $\mb^2 = 0$. 
\end{proof}

The following is a useful consequence of the above lemma, which gives us a lower bound on $g(R)$.

\vtwo

\begin{corollary} \label{C0}Let $(R,\m,\sk)$ be an Artinian local ring with canonical module $\omega$. Then $g(R) \geq \min\{\lambda(R/\ma): \ma \inc R\text{ is a self-dual ideal}\}$. In particular, $g(R) \geq \lambda(R/(\omega^*(\omega)))$.
\end{corollary}

\begin{proof}
Let $S$ be any Gorenstein Artin local ring and $\psi: S \onto R$ be a surjective ring homomorphism. By Lemma \ref{L2}, $\lambda(S) - \lambda(R) \geq \lambda(R/\ma)$ and by Corollary \ref{C0.25}, $\ma$ is a self-dual ideal. Thus $g(R) \geq \min\{\lambda(R/\ma): \ma \simeq \ma^\vee\}$.

The last statement in the corollary follows from Lemma \ref{L1}, since $\ma \inc \omega^*(\omega)$ for every self-dual ideal $\ma$.
\end{proof}

\noindent
Thus, with notation as before, we see that 
\[
\begin{tabular}{|c|}
\hline
\\
$\lambda(R/(\omega^*(\omega))) \leq\hhalf \min\{\lambda(R/\ma): \ma \simeq \ma^\vee\}\hhalf \leq\hhalf  g(R)\hhalf \leq\hhalf  \lambda(R).$\\
\\
\hline
\end{tabular}
\]

\medskip
\noindent
A natural question at this juncture is the following:
\begin{question}\label{Q1}{\rm
Is $\min\{\lambda(R/\ma): \ma \simeq \ma^\vee\} =  g(R)$?
}\end{question}

\vone
\noindent
A stronger question one can ask is:
\begin{question}\label{Q2}{\rm
Given a self-dual ideal $\ma$ in $R$, is there a Gorenstein Artin local ring $S$ such that $\lambda(S) -\lambda(R) = \lambda(R/\ma)$?
}\end{question}

\vone

We answer Question \ref{Q2} in a special case in Theorem \ref{T:1}. The machinery we need to prove the theorem is developed in the next section.
\smallskip

\section{\large An Involution on $\omega^*$}

\begin{remark}\label{R}{\rm
Let $U$, $V$ and $W$ be $R$-modules. Consider the series of natural isomorphisms $$\Hom(U,\Hom(V,W)) \simeq \Hom(U \otimes V, W)$$ $$ \simeq \Hom(V \otimes U, W) \simeq \Hom(V,\Hom(U,W)).$$ Let $f^* \in \Hom(V,\Hom(U,W))$ be the image of a map $f \in \Hom(U,\Hom(V,W))$ under the series of isomorphisms. Then $f(u)(v) = f^*(v)(u)$ for all $u \in U$ and $v \in V$.  

Thus if $U = V$, we get an involution on $\Hom(U, \Hom(U,W))$ induced by the involution $u \otimes v \mapsto v \otimes u$ on $U \otimes U$. In this case, $(f^*)^* = f$.}
\end{remark}

In their paper \cite{HV}, Huneke and Vraciu construct an involution $\adj$ on $\omega^*$ as follows: Let $f \in \Hom_R(\omega,R)$. Fix $u \in \omega$. Consider $\phi_{f,u} : \omega \rar \omega$ defined by $\phi_{f,u}(v) = f(v)\cdot u$. Since $\Hom_R(\omega,\omega) \simeq R$, there is an element $r_{f,u} \in R$ such that $\phi_{f,u}(v) = r_{f,u}\cdot v$.  Define $f^* :\omega \rar R$ by $f^*(u) = r_{f,u}$. We can now define $\adj: \omega^* \lrar \omega^*$ as $\adj(f) = f^*$. One can see that $f^* \in \omega^*$ and that $f^*(u)(v) = f(v)(u)$ for all $u$, $v \in \omega$. Moreover $\adj$ is an involution on $\omega^*$ since $(f^*)^* = f$. 

This involution is the same as the one described in Remark \ref{R} with $U = V = W = \omega$. Note that in this case, $\Hom(\omega,\Hom(\omega,\omega)) \simeq \omega^*$.

The following remarks follow immediately from the definition of $\adj$.

\begin{remark}\label{R2}\hfill{}\\
{\rm
1) $\ker(f) = (0 :_\omega f^*(\omega))$; $f^*(\omega) = (0 :_R \ker(f))$ and vice versa.

\noindent
2) Since $\omega$ is a faithful $R$-module, we see that $f = f^*$ if and only if $f(x)y = f(y)x$ for all $x$, $y \in \omega$, i.e., $f$ satisfies Teter's condition. Thus it follows from (1) that when $f = f^*$, $\ker(f) = 0 :_\omega f(\omega)$ and $f(\omega) = 0:_R \ker(f)$, i.e., $f(\omega)$ is a self-dual ideal in $R$.

\noindent
3) As in the proof of Corollary \ref{C0.25}, $\lambda(\ker(f)) = \lambda(R/f(\omega)) = \lambda(0:_\omega f(\omega))$ by duality. Therefore, if $f(\omega) \cdot \ker(f) = 0$, then $\ker(f) = (0 :_\omega f(\omega)) = \ker(f^*)$. Thus we see that
$\ker(f)\cdot f(\omega) = 0 \Leftrightarrow  \ker(f) = (0 :_\omega f(\omega)) = \ker(f^*) \Leftrightarrow f(\omega) = f^* (\omega).$

In particular, the above equivalent conditions follow from the commutativity condition $f(x)y = f(y)x$ for all $x$, $y \in \omega$ (or equivalently $f = f^*$).
}\end{remark}

\begin{definition}{\rm
Let $(R,\m,\sk)$ be a commutative Noetherian ring and $\ma$ an ideal in $R$. We say that a subring $T$ of $R$ is an algebra retract of $R$ with respect to $\ma$ if the map $\pi \circ i : T \lrar R/\ma$ is an isomorphism, where $i: T \lrar R$ is the inclusion and $\pi: R \lrar R/\ma$ is the natural projection.
}\end{definition}

\begin{remark}\label{R1.5}{\rm
Let $R$, $\ma$ and $T$ be as in the above definition. The condition that $\pi \circ i$ is an isomorphism forces $R = i(T) \oplus \ma$. Identifying $T$ with $i(T)$, we see that $R = T \oplus \ma$ as a $T$-module.}
\end{remark}

\begin{remark}\label{R3}{\rm 
Let $(R,\m,\sk)$ be an Artinian local ring such that $2$ is invertible in $R$. Let $M$ be a finitely generated $R$-module and $\ma$ an ideal in $R$ such that there is a surjective map $f : M \onto \ma$ with $\ma(\ker f) = 0$. Since $f(x) y - f(y)x \in \ker(f)$, for any $w \in M$, $f(w)[f(x)y - f(y)x] = 0$. Thus $$f(w)f(x)y = f(w)f(y)x\text{ for all }w, x, y \in M. \quad\quad\quad(\sharp)$$ 

One can define a multiplicative structure on $M$ as follows: For $x$, $y \in M$, define $x * y = (f(x)y + f(y)x)/2$. This multiplication is associative by ($\sharp$). Thus $M$ is a ring (without a unit) with multiplication induced by $f$.

Further, if $T$ is an algebra retract of $R$ with respect to $\ma$, then one can put a ring structure on $S := T \oplus M$, with addition defined componentwise and multiplication defined as follows: For $(s,x)$, $(t,y)$ in $S$,$$(s,x)(t,y) = (st, sx + ty + x*y) = \left(st, sx + ty + \frac{f(y)x + f(x)y}{2}\right).$$
Note that $S$ is the algebra obtained by attaching a unit to the $T$-algebra $M$ with multiplication induced by $f$.  The ring $S$ is a commutative ring. Moreover, $S$ is an Artinian local ring with maximal ideal $\m_T \oplus M$, where $\m_T = \m \cap T$.
}\end{remark}

The following proposition plays a key role in our proof of Theorem \ref{T:1} and in a corollary (Corollary \ref{C1}) of the main theorem (Theorem \ref{T1}).

\begin{proposition}\label{R4}
Let $(R,\m,\sk)$ be an Artinian local ring with canonical module $\omega$. Let $f \in \omega^*$ be such that $\ker(f) = (0 :_\omega \ma)$ where $\ma := f(\omega)$. Assume that $2$ is invertible in $R$. Then there is a map $h:\omega \lrar R$ satisfying:\vone

\noindent
{\rm 1)} $h(x)y = h(y)x$ for all $x$, $y \in \omega$, i.e., $h$ satisfies Teter's condition.\\
{\rm 2)} $\ker(h) \cap \ma \cdot \omega \inc \ker(f)$. \\
{\rm 3)} $\ker(f) \inc \ker(h) \inc (0 :_\omega {\ma}^2)$, i.e., $\ma^2 \inc h(\omega) \inc \ma$.\\
{\rm 4)} If $(0:_R \ma) \inc \ma^2$, then $\ker(f) = \ker(h)$ {\rm (}or equivalently $h(\omega) = \ma${\rm )}.
\end{proposition}
 
\begin{proof}
Define $h = f + f^*$, where $f^*$ is defined as in Remark \ref{R}. Thus $h = h^*$, i.e., $h$ satisfies Teter's condition. By Remark \ref{R2}.3, this implies that $\ker(h) =( 0:_\omega h(\omega))$. 

We see that by definition of $h$, $\ker(f) \cap \ker(f^*) \inc \ker(h)$. But by Remark \ref{R2}.3 (and the assumption that $\ker(f)\cdot f(\omega) = 0$), $\ker(f) = \ker(f^*)$. Hence $\ker(f) \inc \ker(h)$ giving the first inclusion in (3). The other inclusion in (3) is a consequence of (1) and (2) which can be seen as follows: By (2), $\ma \cdot \ker(h) \inc \ker(f)$. Thus $\ker(h) \inc (\ker(f) :_\omega \ma)$ which gives us $\ker(h) \inc (0 :_\omega \ma^2)$ since $\ker(f) = (0 :_\omega \ma)$ by assumption. The ``i.e.'' part of (3) follows by duality. 

Since $(0:_R \ma) = (0 :_R (\ma \omega))$, $(0 :_R \ma) \inc \ma^2$ gives $(0 :_\omega \ma^2) \inc \ma \omega$. Hence by (2) and (3), $\ker(h) \inc \ker(f)$ proving (4).

In order to prove (2), consider $x_i$, $y_i \in \omega$ such that $\sum f(x_i)y_i \in \ker(h) \cap \ma\omega$. We want to show that $\sum f(x_i)y_i \in \ker(f)$, i.e., $\sum f(x_i)f(y_i) = 0$. Since $\sum f(x_i)y_i \in \ker(h)$, we have $0 = h(\sum f(x_i)y_i) = \sum f(x_i)[f(y_i) + f^*(y_i)] $. Thus, for every $w \in \omega$, $0 = \sum f(x_i)[f(y_i) + f^*(y_i)]w = \sum f(x_i)[f(y_i)w + f(w)y_i] $ and hence by Remark \ref{R3} with $M = \omega$, $2 \sum f(x_i)f(y_i) w = 0$. Since $2$ is invertible in $R$ and $\omega$ is a faithful $R$-module, this forces $\sum f(x_i)f(y_i) = 0$.
\end{proof}

\begin{theorem}\label{T:1}
Let $(R,\m,\sk)$ be an Artinian local ring with canonical module $\omega$. Let $\ma$ be a self-dual ideal in $R$ such that $(0:_R\ma) \subset \ma^2$ and $T$ be an algebra retract of $R$ with respect to $\ma$. Assume further that $2$ is invertible in $R$. Then $g(R) \leq \lambda(R/\ma)$.
\end{theorem} 

\begin{remark}{\rm
When $\ma = \m$, the above hypothesis says that $R$ contains $\sk$ and that $\soc(R) \inc \m^2$. Huneke and Vraciu prove the theorem in this case in \cite{HV}.}
\end{remark}

\noindent
{\it Proof of Theorem} \ref{T:1}. Note that since $\ma$ is self-dual, there is a surjective map $f :\omega \onto \ma$ such that $\ker(f) = (0 :_\omega \ma)$ by Lemma \ref{L1}. We prove the theorem by constructing a Gorenstein Artin local ring $S$ mapping onto $R$ such that $\lambda(S) - \lambda(R) = \lambda(R/\ma)$.

Set $S := T \oplus \omega$. Then $S$ is an Artinian local ring with operations as in Remark \ref{R3}. Define $\phi : S \lrar R$ as $\phi (t,x) = t + f(x)$. Then $\phi$ is a ring homomorphism and it follows from Remark \ref{R1.5} that $\phi$ is surjective. 

We now claim that $S$ is Gorenstein. It is enough to prove that $\lambda(\soc(S)) = 1$. We prove this by showing that $\soc(S) \inc (0:_{\omega}\m)$ which is a one dimensional vector space over $\sk$. 
 
Let $(t,x) \in \soc(S)$ for some $t \in T$ and $x \in \omega$. For each $y \in \omega$, we have $2ty + f(x)y + f(y)x = 0$. Letting $y$ vary over $\ker(f)$, we see that $(2t - f(x)) \in (0:_R \ker(f)) = \ma$. Thus $t \in \ma$ which implies that $t = 0$ by Remark \ref{R1.5}. Now $(0,x)(0,y) = 0$ for all $y \in \omega$ gives $f(x)y + f(y)x = 0$. Thus, if $h: \omega \lrar R$ is defined as in Proposition \ref{R4}, then $h(x)y = 0$ for all $y \in \omega$. Since $\omega$ is a faithful $R$-module, this implies that $x \in \ker(h)$. Therefore, by Proposition \ref{R4}.4, the hypothesis $(0:_R\ma) \subset \ma^2$ gives $x \in \ker(f)$.  

Let $m \in \m$. By Remark \ref{R1.5}, we can write $m = t + a$ for some $t \in \m \cap T$ and $a \in \ma$. Since $x \in \ker(f) = (0:_{\omega} \ma)$, $a \cdot x = 0$. Moreover, since $(0,x) \in \soc(S)$, $(0,x)(t,0) = 0$ gives $t \cdot x = 0$.  Thus $m \cdot x = 0$ for all $m \in \m$ proving the theorem.
$\hfill{\square}$

\section{\large The Main Theorem}

\noindent
{\bf Notation:} We use the following notation in the proof of Theorem \ref{T1}:\\
Let $R$ be any ring and $M$ and $N$ be two $R$-modules. Let $m_i \in M$ and $n_i \in N$ for $1 \leq i \leq n$. We use the notation $(\vec{m}{1}{n}) \overset{\bullet}\otimes (\vec{n}{1}{n})$ to denote $\Sigma(m_i \otimes n_i)$. 

\begin{theorem}\label{T1}
With notation as in Setup {\rm \ref{S1}}, let $\ma$ be an ideal in $T$, $\md \inc \ma$ an ideal generated by a system of parameters such that \\
{\rm a)} there is an injective map $\overline{\md} \overset{\phi}\into \overline{\ma}^\vee$ satisfying $\phi(\overline{x})(\overline{y}) = \phi(\overline{y})(\overline{x})$ for all $x , y \in \md$,\\
{\rm b)} $\mb \inc \ma \md$ and \\
{\rm c)} $(\mb :_T \ma) \inc \md$.\\
Then there is a Gorenstein Artin ring $S$ mapping onto $R$ such that $\lambda(S) - \lambda(R) = \lambda(R/\overline{\ma})$, i.e., $g(R) \leq \lambda(R/\overline \ma)$.
\end{theorem}

\begin{proof}
The map $\phi \in \Hom_R(\overline{\md}, \Hom_R(\overline{\ma},\omega))$ gives a map 
$\tilde{\phi} \in \Hom_R(\overline{\md} \otimes_R \overline{\ma}, \omega)$ defined by $\tilde{\phi}(\overline{x} 
\otimes \overline{y}) = \phi(\overline{x})(\overline{y})$ for any $x \in \md, y \in \ma$, by 
the $\Hom$-$\otimes$ adjointness. The hypothesis implies that $\tilde{\phi}(\overline{x} \otimes \overline{y}) = \tilde{\phi}(\overline{y} \otimes \overline{x})\text{ for }x,y \in \md$. We have a natural map $\pi: \md/\mb \otimes \ma/\mb \lrar  \md\ma/\ma \mb$ defined by $\pi((x + \mb) \otimes (y + \mb)) =  (xy + \ma \mb)$

\noindent 
We claim that:\\
(1) $\tilde{\phi}$ factors through $\md\ma/\ma \mb$, i.e., there is a map $\widehat{\phi}: \md\ma/\ma \mb \rar \omega$ such that $\tilde{\phi} = \widehat{\phi} \circ \pi$,\\
(2) $\mb = (\mc :_T \ma)$, where $\Ker(\widehat{\phi}|_{(\mb/\ma \mb)}) =: \mc/\ma \mb$,\\
(3) $S := T/\mc$ is Gorenstein and\\
(4) $\lambda(S) - \lambda(R) = \lambda(T/\ma)$.

In order to prove (1), it is enough to prove that $\Ker(\pi)$ is generated by elements 
in $\md/\mb \otimes \ma/\mb$ of the form $(x + \mb) \otimes (y + \mb) - (y + \mb) \otimes 
(x + \mb),$ for $x, y \in \md$. In such a case $\tilde{\phi}$ restricts to $\widehat{\phi}$. 

Let $\md$ be minimally generated by the regular sequence \vec{x}{1}{n}. Let $\Sigma(\bar{k_i} \otimes \bar{a_i})$ be an element of $\Ker(\pi)$, where $\bar{x}$ denotes $x + \mb$. Since $k_i \in \md$, without loss of generality we may assume that $\Sigma(\bar{k_i} \otimes \bar{a_i}) = \Sigma_{i=1}^n(\bar{x_i} \otimes \bar{a_i}) \in \Ker(\pi)$. Thus $\pi(\Sigma_{i=1}^n(\bar{x_i} \otimes \bar{a_i})) = \Sigma\overline{x_ia_i} = 0$ in $\md\ma / \ma \mb$. Hence $\Sigma_{i=1}^n x_ia_i \in \ma \mb$. Since $\mb \inc \md\ma$, there are elements $u_{ij},v_j \in \ma$ such that $\Sigma_{i=1}^n x_ia_i = \Sigma_{j=1}^m(\Sigma_{i=1}^n x_iu_{ij})v_j$ in $T$, where $\Sigma (x_iu_{ij}) \in \mb$. Hence $\Sigma_{i=1}^n(\Sigma_j(u_{ij}v_j) - a_i)x_i = 0$ in $T$. Since $\vec{x}{1}{n}$ is a regular sequence in $T$, we can write $$(a_1 - \Sigma_j(u_{1j}v_j), \ldots, a_n - \Sigma_j(u_{nj}v_j)) = 
\Sigma_{i < j} t_{ij} (x_j e_i - x_i e_j)\quad\quad\quad(i)$$ for some $t_{ij} \in T$, where $\{e_i\}_{i = 1}^n$  is the standard basis of $T^n$. Then we have 

$$(\vec{\overline x}{1}{n}) \overset{\bullet}\otimes  (\Sigma_j(\overline{u_{1j}v_j}), \ldots, \Sigma_j(\overline{u_{nj}v_j})) $$ $$= (\Sigma_i(\overline{u_{i1}x_i}), \ldots, \Sigma_i(\overline{u_{im}x_i})) \overset{\bullet}\otimes (\vec{\overline v}{1}{m}) = 0\quad\quad\quad(ii)$$
since $\Sigma_i(u_{ij}x_i) \in \mb$ for each $j$.
Thus, using Equations $(i)$ and $(ii)$, we see that
$$\Sigma_{i=1}^n(\bar{x_i}\otimes \bar{a_i}) = (\vec{\bar{x}}{1}{n}) \overset{\bullet}
\otimes (\vec{\bar{a}}{1}{n}) =$$
$$ (\vec{\bar{x}}{1}{n})  \overset{\bullet}\otimes \Sigma_{i < j} \bar{t_{ij}}(0, \ldots, \bar{x_j},\ldots,-\bar{x_i},\ldots, 0)= \Sigma \bar{t_{ij}} (\bar{x_i} \otimes
 \bar{x_j} - \bar{x_j} \otimes \bar{x_i})$$
verifying (1).

We now have a map $\md\ma/\ma \mb \overset{\widehat{\phi}}\rar \omega$ where $\widehat{\phi}
(\overline{\Sigma a_ib_i}) = \Sigma \phi(\overline{a_i})(\overline{b_i})$. Restrict $\widehat{\phi}$ to $\mb/\ma \mb$, call it $\psi$. Let $\mc \inc T$ be defined by $\mc/\ma \mb = \Ker(\psi)$. Then $(\mc :_T \ma) = \mb$ which can be seen as follows: Let $u \in (\mc:_T \ma)$. Note that the hypothesis $(\mb:_T \ma) \inc \md$ gives $u \in \md$. For any $a \in \ma$, we have $0 = \psi(\overline u \overline a ) = \phi(\bar{u})(\bar{a})$. 
Hence $\phi(\bar{u}) = 0$ in $\overline{\ma}^\vee$. Since $\phi$ is an injective map, $u \in \mb$ as claimed in (2). 

The map $\psi$ induces an inclusion $\mb/\mc \into \omega$. Since $\mb = (\mc :_T\ma)$, $(\mc :_T \m_T) \inc \mb$, i.e., $(\mc:_T \m_T)/\mc \simeq \soc(T/\mc) \inc \mb/\mc$. Therefore the inclusion $(\mc :_T \m_T)/\mc \into \soc(\omega)$, yields $\lambda(\soc(T/\mc)) = 1$ since $\omega$ has a one-dimensional socle. Thus $S := T/\mc$ is Gorenstein proving (3).

Lastly, since $\ma \mb \inc \mc$ and $\mb \inc \ma$, $\mb^2 \inc \mc$. Therefore $(\mb/\mc)^2 = 0$ in $S$ and $R \simeq S/(\mb/\mc)$. Now, by Proposition \ref{KO}.1, $\omega \simeq (0:_S(\mb/\mc)) \simeq \ma/\mc$
since $T/\mc$ is a Gorenstein ring. Thus the image of $\omega$ in $R$ under the natural map from $S$ to $R$ is $\ma/\mb$, and $(\mb/\mc)^2 = 0$ in $S$. Hence by Lemma \ref{L2}, $\lambda(S) - \lambda(R) = \lambda(R/\overline{\ma}) = \lambda(T/\ma)$ which proves (4) and completes the proof of the theorem.
\end{proof}

With notation as in Theorem \ref{T1}, if $\ma = \md$, then condition (c) in the above theorem follows from condition (b). Thus we have the following

\begin{corollary}\label{C2}
With notation as in Setup {\rm \ref{S1}}, let $\ma$ be an ideal in $T$ generated by a system of parameters such that $\mb \inc \ma^2$. Let $\phi: \overline{\ma} \lrar \overline{\ma}^\vee$ be an isomorphism satisfying $\phi(\overline{x})(\overline{y}) = \phi(\overline{y})(\overline{x})$ for all $x , y \in \ma$. Then $g(R) \leq \lambda(R/\overline{\ma})$.
\end{corollary}

We recover Teter's theorem from Corollary \ref{C2} by taking $\ma = \m$. With some additional hypothesis, we see in the next corollary that we can get rid of Teter's condition on the map $\phi$ in Theorem \ref{T1}, just as Huneke and Vraciu did in the case of Teter's theorem.

\begin{corollary}\label{C1}
With notation as in Setup {\rm \ref{S1}}, let $\ma$ be an ideal in $T$, $\md \inc \ma$ an ideal generated by a system of parameters such that \\
{\rm a)} $\overline \ma \simeq \overline\ma^\vee$.\\
{\rm b)} $\mb \inc \ma \md$ and \\
{\rm c)} $(\mb :_T \ma) \inc \md \cap \ma^2$.\\
Further assume that $2$ is invertible in $R$. Then $g(R) \leq \lambda(R/\overline{\ma})$.
\end{corollary}

\begin{proof}
Since $\overline \ma \simeq \overline \ma^\vee$, by Lemma \ref{L1}, $\overline \ma = f(\omega)$ for some $f \in \omega^*$ satisying the condition $\ker(f)\cdot f(\omega) = 0$. By (3), $(\mb :_T \ma) \inc \ma^2$, i.e., $(0 :_R \overline \ma) \inc \overline \ma^2$. Hence by Proposition \ref{R4}, there is a map $h \in \omega^*$ satisfying Teter's condition such that $h(\omega) = \overline \ma$. By Lemma \ref{L1}, since $h$ satisfies Teter's condition, so does the induced isomorphism $\phi :\overline \ma \overset{\sim}\lrar \overline \ma^\vee$.

Thus $\phi$ restricted to $\overline \md$ satisfies condition (1) of Theorem \ref{T1}. Hence the conclusion of the corollary follows from Theorem \ref{T1}. 
\end{proof}

\noindent
By taking $\ma = \md$ in the above corollary, we get the following

\begin{corollary}\label{C3}
With notation as in Setup {\rm \ref{S1}}, let $\ma$ be an ideal in $T$ generated by a system of parameters. Furthermore, assume that $\overline \ma \simeq \overline \ma^\vee$, $2$ is invertible in $R$ and $\mb \inc \ma^3$. Then $g(R) \leq  \lambda(R/\overline{\ma})$.
\end{corollary}

The following is really a corollary, but is important enough to be accorded the status of a theorem.

\begin{theorem}\label{T2}
Let $(R,\m,\sk)$ be an Artinian local ring. Write $R \simeq T/\mb$ where $(T,\m_T,\sk)$ is a regular local ring  and $\mb$ is an $\m_T$-primary ideal. Let $\ \bar{}$ denote going modulo $\mb$. Suppose that $\mb \inc \m_T^6$ and $2$ is invertible in $R$. Then the following are equivalent:\\
{\rm i)} $g(R) \leq 2$.\\
{\rm ii)} There exists a self-dual ideal $\overline \ma \inc R$ such that $\lambda(R/\overline \ma) \leq 2$.
\end{theorem}
\begin{proof}
(i) $\Rightarrow$ (ii) follows from Corollary \ref{C0}.\\
(ii) $\Rightarrow$ (i): If $\lambda(T/\ma) \leq 1$, then by the Huneke-Vraciu Theorem, $g(R) \leq 1$, since $\mb \inc \m_T^6$ implies that $\soc(R) \inc \m^2$. If $\lambda(T/\ma) = 2$, then $\ma$ is generated by a system of parameters and $\mb \inc \m_T^6$ forces $\mb \inc \ma^3$. Hence, by Corollary \ref{C3}, $g(R) \leq 2$.
\end{proof}

\begin{remark}{\rm
Let the hypothesis be as in Theorem \ref{T2}. By combining the conclusions of the Huneke-Vraciu theorem and Theorem \ref{T2}, we see that $\min\{\lambda(R/\ma): \ma \simeq \ma^\vee\} = g(R)$ when either of the two quantities is at most two. Further, it follows from  Theorem \ref{T2} and Corollary \ref{C0} that if $g(R) = 3$, then so is $\min\{\lambda(R/\ma): \ma \simeq \ma^\vee\}$. Thus we see that in this case, Question \ref{Q1} has a positive answer if either $g(R) \leq 3$ or $\min\{\lambda(R/\ma): \ma \simeq \ma^\vee\} \leq 2$.\\
}\end{remark}

\noindent
{\bf Acknowledgement}

\noindent
I would like to thank the referee for valuable suggestions in shortening the proofs and regarding the presentation of this paper. I would also like to acknowledge Srikanth Iyengar for some very interesting discussions and my advisor Craig Huneke for his guidance, support and encouragement.   

\bibliographystyle{amsplain}

\end{document}